\documentclass{amsart}
\usepackage{graphicx}
\usepackage{amssymb,amscd,amsthm,amsxtra}
\usepackage{latexsym}
\usepackage{epsfig}

\newtheorem{thm}{Theorem}[section]
\newtheorem{cor}[thm]{Corollary}
\newtheorem{lem}[thm]{Lemma}

\theoremstyle{definition}
\newtheorem{defn}[thm]{Definition}
\theoremstyle{remark}
\newtheorem{rem}[thm]{Remark}
\numberwithin{equation}{section}

\newcommand{\R}{\mathbb R}

\newcommand{\p}{\partial}

\newcommand{\comment}[1]{}

\begin{document}

\title{A note on higher regularity boundary Harnack inequality}
\author{D. De Silva}
\address{Department of Mathematics, Barnard College, Columbia University, New York, NY 10027}
\email{\tt  desilva@math.columbia.edu}
\author{O. Savin}
\address{Department of Mathematics, Columbia University, New York, NY 10027}\email{\tt  savin@math.columbia.edu}

\thanks{ D.~D.~ and O.~ S.~  are supported by the ERC starting grant project 2011 EPSILON. D.~D. is supported by NSF grant DMS-1301535. O.~S.~ is supported by NSF grant DMS-1200701.}

\begin{abstract}
We show that the quotient of two positive harmonic functions vanishing on the boundary of a $C^{k,\alpha}$ domain is of class $C^{k,\alpha}$ up to the boundary.
\end{abstract}
\maketitle

\section{Introduction}

In this note we obtain a higher order boundary Harnack inequality for harmonic functions, and more generally, for solutions to linear elliptic equations. 

Let $\Omega$ be a $C^{k,\alpha}$ domain in $\R^n$, $k \ge 1$. Assume for simplicity that
$$\Omega := \{(x', x_n) \in \R^n \ | \ x_n > g(x')\}$$ with $$g: \R^{n-1} \to \R, \quad g \in C^{k,\alpha}, \quad \|g\|_{C^{k,\alpha}} \leq 1, \quad  g(0)=0.$$ 

Our main result is the following.
\begin{thm}\label{main} Let $u>0$ and $v$ be two harmonic functions in $\Omega \cap B_1$ that vanish continuously on $\p \Omega \cap B_1$. Assume $u$ is normalized so that $u\left( e_n/2\right)=1,$ then
\begin{equation}\label{BHI}\left\|\frac v u\right \|_{C^{k,\alpha}(\Omega \cap B_{1/2})} \leq C \|v\|_{L^\infty},\end{equation} with $C$ depending on $n,k,\alpha.$
\end{thm}
We remark that if $v>0$, then the right hand side of \eqref{BHI} can be replaced by $\dfrac v u \left(e_n/2\right)$ as in the classic boundary Harnack inequality.

For a more general statement for solutions to linear elliptic equations, we refer the reader to Section 3.

The classical Schauder estimates imply that $u, v$ are of class $C^{k,\alpha}$ up to the boundary. Using that on $\p \Omega$ we have $u=v=0$ and $u_\nu>0$, one can easily conclude that $v / u$ is of class $C^{k-1,\alpha}$ up to the boundary. 

Theorem \ref{main} states that the quotient of two harmonic functions is in fact one derivative better than the quotient of two arbitrary $C^{k,\alpha}$ functions that vanish on the boundary. To the best of our knowledge the result of Theorem \ref{main} is not known in the literature for $k \ge 1$. The case when $k=0$ is well known as boundary Harnack inequality: the quotient of two positive harmonic functions as above must be $C^\alpha$ up to the boundary if $\p \Omega$ is Lipschitz, or the graph of a H\"older function, see \cite{HW, CFMS, JK, HW, F}.

A direct application of Theorem \ref{main} gives smoothness of $C^{1,\alpha}$ free boundaries in the classical obstacle problem without making use of a hodograph transformation, see \cite{KNS, C}.

\begin{cor} Let $\p \Omega \in C^{1,\alpha}$ and let $u$ solve
$$\Delta u= 1\quad \text{in $\Omega$}, \quad u=0, \ \nabla u=0 \quad \text{on $\p \Omega \cap B_1.$}$$ Assume that $u$ is increasing in the $e_n$ direction. Then $\p \Omega \in C^\infty.$
\end{cor}

The corollary follows by repeateadly applying Theorem \ref{main} to the quotient $u_i/u_n$.

Our motivation for the results of this paper comes from the question of higher regularity in thin free boundary problems, which we recently began investigating in \cite{DS}. 

The idea of the proof of Theorem \ref{main} is the following. Let $v$ be a harmonic function vanishing on $\p \Omega$. The pointwise $C^{k+1,\alpha}$ estimate at $0 \in \p \Omega$ is achieved by approximating $v$ with polynomials of the type $x_n P$ with $\deg P=k$.  It turns out that we may use the same approximation if we replace $x_n$ by a given positive harmonic function $u \in C^{k,\alpha}$ that vanishes on $\p \Omega$. Moreover, the regularity of $\p \Omega$ does not play a role since the approximating functions $u \, P$ already vanish on $\p \Omega$. 

In order to fix ideas we treat the case $k=1$ separately in Section 2, and we deal with the general case in Section 3.

\section{The case $k=1$ -- $C^{1,\alpha}$ estimates.}

In this section, we provide the proof of our main Theorem \ref{main} in the case $k=1$. We also extend the result to more general elliptic operators.

Let $\Omega \subset \R^n$ with $\p \Omega \in C^{1,\alpha}$. Precisely,
$$\p \Omega = \{(x', g(x') \ | \ x' \in \R^{n-1}\}, \quad g(0)=0, \quad \nabla_{x'} g(0)=0, \quad \|g\|_{C^{1,\alpha}} \leq 1.$$ Let $u$ be a positive harmonic function in $\Omega \cap B_1$, vanishing continuously on $\p \Omega \cap B_1$. Normalize $u$ so that $u(e_n/2)=1$. Throughout this section, we refer to positive constants depending only on $n,\alpha$ as universal.

\begin{thm}\label{k1} Let v be a harmonic function in $\Omega \cap B_1$ vanishing continuously on $\p \Omega \cap B_1.$ Then,
$$\left\|\frac  v u \right\|_{C^{1,\alpha}(\Omega \cap B_{1/2})} \leq C \|v\|_{L^\infty(\Omega \cap B_1)}$$ with $C$ universal.
\end{thm}

First we remark that from the classical Schauder estimates and Hopf lemma, $u$ satisfies
\begin{equation}\label{prop_u_1} u \in C^{1,\alpha}, \quad \|u\|_{C^{1,\alpha}(\Omega \cap B_{1/2})} \leq C, \quad u_{\nu} > c >0 \quad \text{on $\p \Omega \cap B_{1/2}$.}
\end{equation}

Thus, after a dilation and multiplication by a constant we may assume that 
\begin{equation}\label{dil}
\|g\|_{C^{1,\alpha}(B_1)} \leq \delta, \quad \nabla u(0)=e_n, \quad [\nabla u]_{C^\alpha} \leq \delta, 
\end{equation}
where the constant $\delta$ will be specified later.

We claim that Theorem \ref{k1} will follow, if we show that there exists a linear function \begin{equation}\label{P}P(x) = a_0 + \sum_{i=1}^n a_i x_i, \quad a_n=0\end{equation} such that 
\begin{equation}\label{est1} \left|\frac v u(x) - P(x)\right| \leq C |x|^{1+\alpha}, \quad x \in \Omega \cap B_1\end{equation}  for $C$ universal. 

To obtain \eqref{est1}, we prove the next lemma. 

\begin{lem} \label{imp}Assume that, for some $r \leq 1$ and $P$ as in \eqref{P} with $|a_i| \leq 1,$
$$\|v - uP\|_{L^\infty(\Omega \cap B_r)} \leq r^{2+\alpha}.$$ Then, there exists a linear function $$\bar P(x) = \bar a_0 + \sum_{i=1}^n \bar a_i x_i, \quad \bar a_n=0$$ such that
$$\|v - u \bar P\|_{L^\infty(\Omega \cap B_{\rho r})} \leq (\rho r)^{2+\alpha},$$ for some $\rho >0$ universal, and $$\|P-\bar P\|_{L^\infty(B_r)} \leq C r^{1+\alpha}, $$ with $C$ universal. 
\end{lem}

\begin{proof} We write
$$v(x) = u(x)P(x) + r^{2+\alpha} \tilde v\left(\frac x r \right), \quad x \in \Omega \cap B_r,$$ with
$$\|\tilde v\|_{L^\infty(\tilde \Omega \cap B_1)} \leq 1, \quad \tilde \Omega := \frac 1 r \Omega.$$
Define also,
$$\tilde u(x) := \frac{u(rx)}{r}, \quad x\in \tilde \Omega \cap B_1.$$

We have,
$$0= \Delta v = \Delta (uP) + r^\alpha \Delta \tilde v\left(\frac x r \right), \quad x\in \Omega \cap B_r,$$
and
$$\Delta (uP) = 2 \nabla u \cdot \nabla P = 2 \sum_{i=1}^{n-1} a_i u_i, \quad x \in \Omega \cap B_r.$$
Moreover, from \eqref{dil} we have
$$\|\nabla u - e_n\|_{L^\infty (\Omega \cap B_r)} \leq \delta r^\alpha.$$
Thus, $\tilde v$ solves
\begin{equation}\label{tildev}|\Delta \tilde v | \leq 2\delta \quad \text{in $\tilde \Omega \cap B_1$}, \quad \tilde v = 0 \quad \text{on $\p \tilde \Omega \cap B_1,$}\end{equation}
and $$\|\tilde v\|_{L^\infty(\tilde \Omega \cap B_1)} \leq 1.$$
Hence, as $\delta \to 0$ (using also \eqref{dil}) $\tilde v$ must converge (up to a subsequence) uniformly to a solution $v_0$ of
$$\Delta v_0 = 0 \quad \text{in $B_1^+$}, \quad  v_0 = 0 \quad \text{on $\{x_n=0\} \cap \bar B_1^+$}$$
and 
$$|v_0| \leq 1 \quad \text{in $B_1^+$}. $$ 

Such a $v_0$ satisfies,
$$\|v_0 - x_n Q\|_{L^\infty(B_\rho^+)} \leq C \rho^3 \leq \frac 1 4 \rho^{2+\alpha},$$
 for some $\rho$ universal and $Q = b_0 + \sum_{i=1}^n b_i x_i, |b_i| \leq C.$ Notice that $b_n=0$ since $x_n Q$ is harmonic.
 
 By compactness, if $\delta$ is chosen sufficiently small, then 
 $$\|\tilde v - x_n Q\|_{L^\infty(\tilde \Omega \cap B_\rho)} \leq \frac 1 2 \rho^{2+\alpha}.$$
 From \eqref{dil}, $$|\tilde u - x_n| \leq \delta$$ thus
 $$\|\tilde v - \tilde u Q\|_{L^\infty(\tilde \Omega \cap B_\rho)} \leq \rho^{2+\alpha}$$ 
from which the desired conclusion follows by choosing 
$$\bar P(x) = P(x) + r^{1+\alpha} Q\left(\frac x r \right).$$
\end{proof}

\begin{rem}\label{betterest} Notice that, from boundary Harnack inequality, $\tilde v$ satisfies (see \eqref{tildev} and recall that $u(\frac 1 2 e_n) =1$)
$$|\tilde v | \leq C \tilde u \quad \text{in $\tilde \Omega \cap B_{1/2}$},$$ with $C$ universal. Thus our assumption can be improved in $B_{r/2}$ to 
$$|v(x) - uP(x)| \leq C u(x)r^{1+\alpha} \quad \text{in $\Omega \cap B_{r/2}$.}$$

Moreover, 
$$\left[\frac{\tilde v}{\tilde u}\right]_{C^{1,\alpha}(\tilde \Omega \cap B_{1/4}(\frac 12 e_n))} \leq C$$ since $\tilde u$ is bounded below in such region. This, together with the identity
$$\frac{v}{u} = P + r^{1+\alpha} \frac{\tilde v}{\tilde u}\left(\frac x r\right)  \quad x \in \Omega \cap B_r$$
implies
\begin{equation}\label{**}  \left[\nabla \left(\frac{v}{u}\right)\right]_{C^{\alpha}(\Omega \cap B_{r/4}(\frac r 2 e_n))}=\left[\frac{\tilde v}{\tilde u}\right]_{C^{1,\alpha}(\tilde \Omega \cap B_{1/4}(\frac 12 e_n))} \leq C.\end{equation}
\end{rem}

\

\textit{Proof of Theorem $\ref{k1}.$} After multiplying $v$ by a small constant, the assumptions of the lemma are satisfied with $P=0$ and $r=r_0$ small. Thus, if we choose $r_0$ small universal, we can apply the lemma indefinitely and obtain a limiting linear function $P_0$ such that $$|v - u P_0| \leq Cr^{2+\alpha}, \quad r \leq r_0.$$ In fact, from Remark \ref{betterest} we obtain
$$\left|\frac v u - P_0\right| \leq C |x|^{1+\alpha}$$ which together with \eqref{**} gives the desired conclusion. 
\qed

\

It is easy to see that our proof holds in greater generality. For example, if $v$ solves $\Delta v=f  \in C^\alpha$ in $\Omega \cap B_1$  and vanishes continuously on $\p \Omega \cap B_1,$ then we get
$$\left\|\frac  v u \right\|_{C^{1,\alpha}(\Omega \cap B_{1/2})} \leq C( \|v\|_{L^\infty} + \|f\|_{C^\alpha}).$$ To obtain this estimate it suffices to take in Lemma \ref{imp} linear functions $P(x)= a_0 + \sum_{i=1}^n a_i x_i$ satisfying $2a_n u_n(0) = f(0).$ In fact, the following more general Theorem holds. 

\begin{thm}\label{gen} Let $$\mathcal L u := Tr(A \,  D^2 u) + b\cdot \nabla u + c \, u,$$ with $A \in C^\alpha, b, c \in L^\infty$ and $$\lambda I \leq A \leq \Lambda I, \quad \|A\|_{C^\alpha}, \|b\|_{L^\infty}, \|c\|_{L^\infty} \leq \Lambda.$$ Assume $$\mathcal L u = 0, \, u>0 \quad \text{in $\Omega \cap B_1$}, \quad u=0 \quad \text{on $\p \Omega \cap B_1$}$$ and 
$$\mathcal Lv = f \in C^\alpha \quad \text{in $\Omega \cap B_1,$} \quad v=0\quad \text{on $\p \Omega \cap B_1$}.$$ Then, if $u$ is normalized so that $u(\frac 1 2 e_n)=1$
$$\left\|\frac v u \right\|_{C^{1,\alpha}(\Omega \cap B_{1/2})} \leq C(\|v\|_{L^\infty} + \|f\|_{C^\alpha}) $$ with $C$ depending on $\alpha, \lambda, \Lambda$ and $n$.
\end{thm}

\begin{rem} We emphasize that the conditions on the matrix  $A$ and the right hand side $f$ are those that guarantee interior $C^{2,\alpha}$ Schauder estimates. However the conditions on the domain $\Omega$ and the lower order coefficients $b, c$ are those that guarantee interior $C^{1,\alpha}$ Schauder estimates. \end{rem}

\begin{rem} The theorem holds also for divergence type operators
$$Lu=\text{div}(A \nabla u + b u), \quad \quad A \in C^\alpha, \quad b \in C^ \alpha.$$

\end{rem}

The proof of Theorem \ref{gen} follows the same argument of Theorem \ref{k1}. For convenience of the reader,  we give a sketch of the proof.

\

\textit{Sketch of the proof of Theorem $\ref{gen}$.} After a dilation we may assume that \eqref{dil} holds and also
\begin{equation}\label{dil2}  A(0) = I, \quad \max\{[A]_{C^\alpha},\|b\|_{L^\infty},\|c\|_{L^\infty},  [f]_{C^\alpha}\} \leq \delta
\end{equation}
with $\delta$ to be chosen later. Again, it suffices to show the analogue of Lemma \ref{imp} in this context, with the $x_n$ coefficient of $P$ and $\bar P$ satisfying $$2a_n = 2 \bar a_n = f(0).$$ 
Define $\tilde v$ as before. Then
$$f = \mathcal L v = \mathcal L(uP) + r^\alpha \tilde{\mathcal L} \tilde v \left(\frac x r\right) \quad x\in \Omega \cap B_r$$
with
$$ \tilde{\mathcal L} \tilde v := Tr(\tilde A \, D^2 \tilde v) + r \tilde b \cdot \nabla \tilde v + r^2 \tilde c \, \tilde v,$$

$$\tilde A(x) = A(rx), \quad \tilde b (x) = b(rx), \quad \tilde c (x) = c(rx), \quad x \in \tilde \Omega \cap B_1.$$

On the other hand,
$$\mathcal L(uP)= (\mathcal L u)P + 2 (\nabla u)^T A \nabla P + u \, b \cdot \nabla P$$
thus, using \eqref{dil}-\eqref{dil2} and the fact that $2a_n=f(0)$
$$|\mathcal L(uP) - f | \leq C \delta r^\alpha, \quad x \in \Omega \cap B_r.$$ From this we conclude that $$|\tilde{\mathcal{L}} \tilde v| \leq C \delta \quad \text{in $\tilde \Omega \cap B_1$}$$
and we can argue by compactness exactly as before. \qed

\section{The general case, $k \geq 2.$}

Let $\Omega \subset \R^n$ with $\p \Omega \in C^{k,\alpha}$. Precisely,
$$\p \Omega = \{(x', g(x') \ | \ x' \in \R^{n-1}\}, \quad g(0)=0, \quad \nabla_{x'} g(0)=0, \quad \|g\|_{C^{k,\alpha}} \leq 1.$$

\begin{thm}\label{genk2} Let $$\mathcal L u := Tr(A D^2 u) + b \cdot \nabla u + cu$$ with  $$\lambda I \leq A \leq \Lambda I,$$ and $$\max\{  \|A\|_{C^{k-1, \alpha}}, \|b\|_{C^{k-2,\alpha}}, \|c\|_{C^{k-2,\alpha}} \}\leq \Lambda.$$ Assume \begin{equation}\label{u}\mathcal Lu = 0, u>0 \quad \text{in $\Omega \cap B_1$}, \quad u=0 \quad \text{on $\p \Omega \cap B_1$}\end{equation} and 
\begin{equation}\label{v}\mathcal Lv = f \in C^{k-1,\alpha} \quad \text{in $\Omega \cap B_1,$} \quad v=0\quad \text{on $\p \Omega \cap B_1$}.\end{equation} Then, if $u$ is normalized so that $u(\frac 1 2 e_n)=1$
$$\left\|\frac v u \right \|_{C^{k,\alpha}(\Omega \cap B_{1/2})} \leq C(\|v\|_{L^\infty} + \|f\|_{C^{k-1,\alpha}}) $$ with $C$ depending on $k, \alpha, \lambda, \Lambda$ and $n$.
\end{thm}

From now on, a positive constant depending on $n, k, \alpha, \lambda, \Lambda$ is called universal.

\begin{rem} If we are interested only in $C^{k,\alpha}$ estimates for $\dfrac v u$ on $\p \Omega \cap B_{1/2}$, then the regularity assumption on $c$ can be weakened to $\|c\|_{C^{k-3, \alpha}} \leq \Lambda.$
\end{rem}
If $u$ and $v$ solve \eqref{u}-\eqref{v} respectively, the rescalings 
$$\tilde u(x) = \frac{1}{r_0} u(r_0 x), \quad \tilde v(x) = \frac{1}{r_0} v(r_0 x)$$ satisfy the same problems with $\Omega, A, b, c$ and $f$ replaced by
\begin{align*} &\tilde \Omega = \frac{1}{r_0} \Omega, \quad \tilde A(x) = A(r_0x), \quad \tilde b(x)=r_0 b(r_0 x), \\  &\tilde c(x) = r_0^2 c(r_0 x),  \quad \tilde f(x)= r_0 f(r_0 x).\end{align*}
Thus, as in the case $k=1$, we may assume that 
$$\nabla u(0)=e_n, \quad A(0)=I$$
and that the following norms are sufficiently small:
\begin{equation}\label{small} 
\max\{\|g\|_{C^{k,\alpha}}, \|A-I\|_{C^{k-1,\alpha}}, \|b\|_{C^{k-2,\alpha}}, \|c\|_{C^{k-2,\alpha}}, \|f\|_{C^{k-1,\alpha}}, \|u-x_n\|_{C^{k,\alpha}} \} \leq \delta,
\end{equation} with $\delta$ to be specified later.

The proof of Theorem \ref{genk2} is essentially the same as in the case $k=1$. However, we now need to work with polynomials of degree $k$ rather than linear functions. 

We introduce some notation. A polynomial $P$ of degree $k$ is denoted by
$$P(x)= a_m x^m, \quad m=(m_1,m_2,\ldots, m_n), |m|=m_1+\ldots+m_n,$$
with the $a_m$ non-zero only if $m \geq 0$ and $|m| \leq k.$
We use here the summation convention over repeated indices and the notation 
$$x^{m}=x_1^{m_1}\ldots x_{n}^{m_n}.$$

Also, in what follows, $\bar i $ denotes the multi-index with $1$ on the $i$th position and zeros elsewhere and  $\|P\|=\max|a_m|$.

Given $u$ a solution to \eqref{u}, we will approximate  a solution $v$ to \eqref{v} with polynomials $P$
such that $\mathcal L(uP)$ and $f$ are tangent at 0 of order $k-1$.

Below we show that the coefficients of such polynomials must satisfy a certain linear system.

Indeed,
$$\mathcal L(uP) = (\mathcal L u)P + 2 (\nabla u)^T A \nabla P + u \, tr(A D^2P) + u \, b \cdot \nabla P.$$
Since $\mathcal L u=0,$ we find 
$$\mathcal L(uP) = g^i P_i +  g^{ij}P_{ij}, \quad g^i \in C^{k-2,\alpha},  \, \, g^{ij} \in C^{k-1,\alpha}.$$
Using the first order in the expansions below ($l.o.t$ = lower order terms),
$$A=I + l.o.t., \quad u=x_n + l.o.t, \quad \nabla u= e_n + l.o.t., $$ we write each $g^i, g^{ij}$ as a sum of a polynomial of degree $k-1$ and a reminder of order $O(|x|^{k-1+\alpha})$. We find 
$$g^i=2\delta_{in} + l.o.t, \quad g^{ij}=\delta_{ij} x_n + l.o.t.$$

In the case $P=x^m$ we obtain 
\begin{align*} \mathcal L(u \, x^m) = & \, m_n(m_n+1) x^{m-\bar n} + \sum_{i \neq n} m_i(m_i-1) x^{m-2\bar i +\bar n} \\ &+ c_l^m x^l + w_m(x), \end{align*}
with \begin{equation} \label{c1}c_l^m \neq 0 \quad \text{only if $|m| \leq |l| \leq k-1$}, \quad \quad \mbox{and} \quad w_m=O(|x|^{k-1+\alpha}).\end{equation}
Also in view of \eqref{small} 
\begin{equation}\label{c2}|c_l^m| \leq C \delta, \quad |w_m| \le C \delta |x|^{k-1+\alpha}, \quad \|w_m\|_{C^{k-2,\alpha}(B_r)} \leq C\delta r.\end{equation}

Thus, if $P=a_m x^m, $ with $\|P\| \leq 1$ then 
$$\mathcal L(uP) = R(x) + w(x), \quad R(x) = d_lx^l, \quad \deg R=k-1,$$ with $w$ as above and the coefficients of $R$ satisfying 
\begin{equation}\label{Q}
d_l = (l_n +1)(l_n+2)a_{l+\bar n} + \sum_{i \neq n} (l_i+1)(l_i+2) a_{l+2\bar i -n} +c_l^m a_m.
\end{equation}

\begin{defn} We say that $P$ is an approximating polynomial for $v/u$ at $0$ if the coefficients $d_l$ of $R(x)$ coincide with the coefficients of the Taylor polynomial of order $k-1$ for $f$ at 0.
\end{defn}

We think of \eqref{Q} as an equation for $a_{l+\bar n}$ in terms of $d_l$ and a linear combination of $a_m$'s with either $|m| < |l| +1$ or when $|m|= l+1$ with $m_n < l_n+1.$
Thus the $a_m$'s are uniquely determined from the system \eqref{Q} once $d_l$ and $a_m$ with $m_n=0$ are given.

The proof of Theorem \ref{genk2} now follows as in the case $k=1$, once we establish the next lemma.

\begin{lem} Assume that for some $r \leq 1$ and an approximating polynomial $P$ for $v/u$ at $0$,  with $\|P\| \le 1$, we have
$$\|v-uP\|_{L^\infty(\Omega \cap B_r)} \leq r^{k+1+\alpha}.$$ Then, there exists an approximating polynomial $\bar P$ for $v/u$ at $0$, such that 
$$\|v-u\bar P\|_{L^\infty(\Omega \cap B_{\rho r})} \leq (\rho r)^{k+1+\alpha}$$
for $\rho>0$ universal, and 
$$\|P-\bar P\|_{L^\infty(B_r)} \leq Cr^{k+\alpha},$$ with $C$ universal.
\end{lem}
\begin{proof} 
We write
$$v(x) = u(x)P(x) + r^{k+1+\alpha} \tilde v\left(\frac x r \right), \quad x \in \Omega \cap B_r,$$ with
$$\|\tilde v\|_{L^\infty(\tilde \Omega \cap B_1)} \leq 1, \quad \tilde \Omega := \frac 1 r \Omega.$$
Define also,
$$\tilde u(x) := \frac{u(rx)}{r}, \quad x\in \tilde \Omega \cap B_1.$$

Then
$$f = \mathcal L v = \mathcal L(uP) + r^{k+\alpha-1} \tilde{\mathcal L} \tilde v \left(\frac x r\right) \quad x\in \Omega \cap B_r$$
with
$$ \tilde{\mathcal L} \tilde v := Tr(\tilde A \, D^2 \tilde v) + r \tilde b \cdot \nabla \tilde v + r^2 \tilde c \, \tilde v,$$

$$\tilde A(x) = A(rx), \quad \tilde b (x) = b(rx), \quad \tilde c (x) = c(rx), \quad x \in \tilde \Omega \cap B_1.$$

Using that $P$ is approximating, we conclude that
\begin{equation}\label{tildeL}
\tilde{\mathcal L} \tilde v = \tilde w \quad \text{in $\tilde \Omega \cap B_1$, } \quad \tilde v=0 \quad \text{on $\p \tilde \Omega \cap B_1,$}
\end{equation}
with 
$$\|\tilde v \|_{L^\infty} \leq 1, \quad \|\tilde w\|_{C^{k-2,\alpha}} \leq C\delta.$$
By compactness $\tilde v \to v_0$ with $v_0$ harmonic. Thus we find,
$$\|\tilde v - x_n Q\|_{L^\infty(\tilde \Omega \cap B_\rho)} \leq C \rho^{k+2} \leq \frac 1 2 \rho^{k+\alpha -1}, \quad \deg Q=k, \quad \quad \|Q\| \le C,$$
with $x_n Q$ a harmonic polynomial and $\rho$ universal.
Thus,
$$\|v - u(P+r^{k+\alpha}Q(\frac x r))\|_{L^\infty(\Omega \cap B_{\rho r})} \leq \frac 1 2 (\rho r)^{k+\alpha-1}.$$
However $P+r^{k+\alpha}Q(\frac x r)$ is not approximating for $v/u$ at $0$, and we need to modify $Q$ into a slightly different polynomial $\bar Q$. 

We want the coefficients $\bar q_l$ of $\bar Q$ to satisfy (see \eqref{Q})
\begin{align}\label{eqq}
0&=(l_n+1)(l_n+2)\bar q_{l+\bar n} + \sum_{i \neq n } (l_i +1)(l_i+2)\bar q_{l+2\bar i-\bar n}+
\bar c_l^m \bar q_m, \end{align}
with (see \eqref{c1}-\eqref{c2})
$$\bar c_l^m=r^{|l|+1-m} c_l^m, \quad |\bar c_l^m| \leq C \delta.$$

Moreover, since in the flat case i.e. $A=I$, $u=x_n$ and $g$, $b$, $c$, $f$ all vanishing, $Q$ is approximating for $v_0/x_n$ at 0, the coefficients of $Q$ satisfy the system \eqref{Q} with $c_l^m=0$ and $d_l=0$, i.e.
$$0=(l_n+1)(l_n+2) q_{l+\bar n} + \sum_{i \neq n } (l_i+1)(l_i+2)q_{l+2 \bar i-\bar n}.$$
Thus, by subtracting the last two equations,  the coefficients of $Q-\bar Q$ solve the system \eqref{eqq} with left hand side bounded by $C\delta$, and we can find $\bar Q$ such that 
$$\|Q-\bar Q\|_{L^\infty(B_1)} \leq C\delta.$$
\end{proof}

\end{document}